\documentclass{compositio}

\usepackage{amsmath,amsfonts,yhmath,hyperref}
\usepackage{amsthm}
\usepackage{amssymb}
\usepackage[latin9]{inputenc}
\usepackage[nottoc,notlof,notlot]{tocbibind}
\usepackage{mathtools}
\usepackage{enumitem}
\usepackage{tikz-cd}
\usepackage{multirow}
\usepackage{dsfont}

\usepackage{float}
\usepackage{array}
\usepackage{mathrsfs}
\usepackage{appendix}

\usepackage{pgf,tikz}

\usetikzlibrary{arrows}
\usetikzlibrary[patterns]
\usetikzlibrary{shapes}
\usetikzlibrary{plotmarks}

\newcommand{\PP}{\mathbb{P}}

\newcommand{\RR}{\mathbb{R}}

\newcommand{\CCC}{\mathscr{C}}

\newcommand{\TTT}{\mathscr{T}}

\renewcommand{\P}{\mathcal{P}}

\newcommand{\modulo}{\ \mathrm{mod}\ }

\newtheorem{theo}{Theorem}[section]
\newtheorem*{theom}{Theorem}
\newtheorem{prop}[theo]{Proposition}

\newtheorem{lem}[theo]{Lemma}

\theoremstyle{definition}

\theoremstyle{remark}
\newtheorem{remark}[theo]{Remark}

\theoremstyle{remark}
\newtheorem{example}[theo]{Example}
\newenvironment{expl}[1]{
    \begin{example}#1}{
    \xqed{\lozenge}\end{example}
}

\newcommand{\xqed}[1]{
    \leavevmode\unskip\penalty9999 \hbox{}\nobreak\hfill
    \quad\hbox{\ensuremath{#1}}}

\begin{document}
 
%\renewcommand{\qedsymbol}{\filledbox}
%Good resources for looking up how to do stuff:
%Binary operators: http://www.access2science.com/latex/Binary.html
%General help: http://en.wikibooks.org/wiki/LaTeX/Mathematics
%Or just google stuff
 
\title{On the double tangent of projective closed curves}
\author{Thomas Blomme}
\email{thomas.blomme@unige.ch}
\address{Universit\'e de Neuch\^atel, rue \'Emile Argan 11,
Neuch\^atel 2000, Switzerland}

\begin{abstract}
We generalize a previous result by Fabricius-Bjerre \cite{fabricius1962double} from curves in $\RR^2$ to curves in $\RR P^2$. Applied to the case of real algebraic curves, this recovers the signed count of bitangents of quartics introduced by Larson-Vogt \cite{larsonvogt2021enriched} and proves its positivity, conjectured by Larson-Vogt. Our method is not specific to quartics and applies to algebraic curves of any degree.
\end{abstract}

\maketitle

\tableofcontents

\section{Introduction}

%\subsection{Curves and classical Fabricius-Bjerre}
We consider curves in $\RR^2$ and $\RR P^2$. By \textit{curve} we mean a closed immersed $C^1$ curve which is the union of finitely many strictly convex arcs. Such a curve has a finite number of \textit{flexes} (points where the curve passes through its tangent), \textit{nodes}  (also called double points or self-intersection) and double tangents, also called \textit{bitangents}, which are lines tangent at two points of the curve. Throughout the paper, we assume all our curves to be \textit{generic} in the sense that double points are simple, no tangent at a flex is tangent elsewhere, and a double tangent has only two tangency points. The normalization of the curve (desingularizing the double points) may not be connected, and its connected components are called \textit{components} of the curve.

In 1962, Fabricius-Bjerre \cite{fabricius1962double} introduced a signed count of bitangents for such curves in $\RR^2$. Let $C\subset\RR^2$ be a curve and $D$ be a bitangent to $C$. There are two types of bitangents according to the relative position of the curve near the tangency points (See Figure \ref{fig-type-bitangents}):
    \begin{itemize}[label=$\circ$]
        \item \textit{exterior} (type $T$) if the arcs of $C$ near the tangency points lie on the same side of $D$,
        \item \textit{interior} (type $S$) if the arcs of $C$ near the tangency points lie on opposite sides of $D$.
    \end{itemize}
The Fabricius-Bjerre sign of a bitangent is $+$ for type $T$ and $-$ for type $S$. If $C$ has $i(C)$ flexes and $n(C)$ self-intersections, Fabricius-Bjerre \cite{fabricius1962double} states that the signed count $\sigma(C)$ of bitangents satisfies
$$\sigma(C)=n(C)+\frac{i(C)}{2}.$$
Although the proof of \cite{fabricius1962double} seems to assume the curve has only one component, such an assumption is actually not necessary.

\begin{figure}
\centering
\begin{tabular}{cc}
\begin{tikzpicture}[x=0.75cm,y=0.75cm]
\clip (-1,-0.2) rectangle (5.2,1.2);
\draw[blue] (0,0) to (5,0);
\draw (1,1) circle(1);
\draw (4,1) circle(1);
\end{tikzpicture} & \begin{tikzpicture}[x=0.75cm,y=0.75cm]
\clip (-1,-1.2) rectangle (5.2,1.2);
\draw[blue] (0,0) to (5,0);
\draw (1,-1) circle(1);
\draw (4,1) circle(1);
\end{tikzpicture} \\
$(T)$, exterior bitangent, $+$ & $(S)$, interior bitangent, $-$ \\
\end{tabular}
\caption{\label{fig-type-bitangents}Types of bitangents}
\end{figure}

%\subsection{Projective case}
Our main result is to provide a generalization of the Fabricius-Bjerre count to the projective setting. To extend the sign definition, we choose a line $L_\infty\subset\RR P^2$ and consider curves generically transverse to $L_\infty$: no line passing through a point of $C\cap L_\infty$ is a bitangent or tangent at a flex, which means $L_\infty$ is transverse to $C$ union its bitangents and tangent at flexes. Let $\TTT_\infty(C)$ be the set of tangents to $C$ at points of $C\cap L_\infty$. As $\RR P^2\backslash L_\infty=\RR^2$, we choose as sign for the bitangents the well-defined Fabricius-Bjerre sign in this affine chart. Notice it depends on the choice of $L_\infty$.

\begin{theom}\ref{theo-projective-case}
Let $C$ a curve in $\RR P^2$ transverse to $L_\infty$ with $i(C)$ flexes, $n(C)$ nodes and $a$ intersection points with $L_\infty$. Then, we have the following signed count of bitangents:
$$\sigma(C)=n(C)+\frac{i(C)}{2}+\frac{a(a-2)}{2}-\sum_{T\in\TTT_\infty(C)} (|C\cap T|-1).$$
\end{theom}

The case of \cite{fabricius1962double} corresponds to curve such that $C\cap L_\infty=\emptyset$ (i.e. $a=0$).

%\subsection{Applications to algebraic curves}
If $\CCC$ is a real algebraic curve, Pl\"ucker \cite{Plu34,Plu39} already knew the number of complex bitangents, which only depends on the degree. However, the number of real bitangents depends on the curve. One may hope that there exists some suitable signed counts of the bitangents that is to some extent invariant. In \cite{blomme2024bitangents}, Blomme-Brugall\'e-Garay proved the existence of a signed count only depending on the topology of the real part for curves of even degree. Prior to the latter, Larson-Vogt proved the existence of a different signed count for quartics, with some conjectures concerning so-called \textit{quadratic enrichment} to more general fields. The Fabricius-Bjerre sign actually coincides with the Larson-Vogt sign \cite{larsonvogt2021enriched} used to count real bitangents to a real quartic curve. In particular, we have surprisingly two different signed counts for curves in $\RR^2$: the one from \cite{blomme2024bitangents} and the one from \cite{fabricius1962double}.

In the real algebraic setting, bitangents can be \textit{split} or \textit{non-split} according to whether the tangency points are real or complex conjugate. Given a smooth curve $\CCC$, Klein's formula \cite{klein1876formula} provides a relation between the number of non-split bitangents $t_0(\CCC)$ and the number of real flexes. It is possible to apply Theorem \ref{theo-projective-case} coupled to Klein's formula to a real algebraic curve transverse to a line $L_\infty$ and get a signed count of the bitangents: $\rho(\CCC)=t_0(\CCC)+\sigma(\RR\CCC)$.

\begin{theom}\ref{theo-algebraic-case}
Let $\CCC$ be a smooth real algebraic curve of degree $d$ in $\RR P^2$ intersecting $L_\infty$ transversally at exactly $a$ points. Then, we have the following signed count:
$$\rho(\CCC)=\frac{d(d-2)}{2}+\frac{a(a-2)}{2}-\sum_{T\in\TTT_\infty(\RR\CCC)}(|\RR\CCC\cap T|-1).$$
Furthermore, the integer $\rho(\CCC)$ is even and non-negative
\end{theom}

In the case of quartics non-intersecting $L_\infty$, we recover the signed count from \cite{larsonvogt2021enriched}. Furthermore, the non-negativity statement proves \cite[Conjecture 2]{larsonvogt2021enriched} for any degree, already proven for quartics by Kummer-McKean \cite{kummer2024bounding} through different methods.

It would be very interesting to see if \textit{quadratic enrichments} exist for fields other than $\RR$, where a sined count is essentially a signed count. Actually, \cite{larsonvogt2021enriched} already proposes an enriched way of counting bitangents, though it depends on the choice of a line $L_\infty$. Theorem \ref{theo-algebraic-case} suggests such a count should be related to the intersection points between the curve and its tangents to points intersecting $L_\infty$. Another direction would be to find a quadratic enrichment of the signed count of \cite{blomme2024bitangents}, which has the advantage of not depending on the choice of a line, though it is not clear how to do so.

\medskip

\textit{Acknowledgments.} The author is supported by the SNF grant 204125 and would like to thank Erwan Brugall\'e for suggesting to work on bitangents and write this paper in the first place, as well as commenting on a first version.

\section{Classical Fabricius-Bjerre Theorem}	
\label{sec-classical-FB}

Let $C$ be a closed curve in $\RR^2$ with components $C_j$. For each component, we choose a regular parametrization $\gamma_j:I_j\to\RR^2$ inducing an orientation. We denote the tangent at $\gamma_j(t_j)$ by $D(\gamma_j(t_j))=\gamma_j(t_j)+\RR\gamma'_j(t_j)$. We also consider the half-tangents
$$D_+(\gamma_j(t_j))=\gamma_j(t_j)+\RR_{>0}\gamma'_j(t_j) \text{ and }D_-(\gamma_j(t_j))=\gamma_j(t_j)+\RR_{<0}\gamma'_j(t_j).$$
They are switched when the orientation of $C$ is reversed. Let $t(C)$ (resp. $s(C)$) be the number of bitangents of type $T$ (resp. $S$), so that the signed count is $\sigma(C)=t(C)-s(C)$.

The following Theorem is first due to Fabricius-Bjerre \cite{fabricius1962double} in the case where the curve $C$ has a unique component. For convenience of the reader, we recall its proof. We also notice that the proof also works for curves with possibly several components.

\begin{theo}\label{theo-affine-case}
Let $C$ be a curve in $\RR^2$ with $n(C)$ nodes and $i(C)$ flexes. We have the following identity:
$$\sigma(C)=n(C)+\frac{i(C)}{2} .$$
\end{theo}

\begin{proof}
On each component $C_j$ of $C$, we consider the function
$$f_j(t_j)=|D_+(\gamma_j(t_j))\cap C|-|D_-(\gamma_j(t_j))\cap C|,$$
giving the difference between the number of intersection points of the curve and the two half-tangents.

If $D(\gamma_j(t_j))$ is transverse to $C$ outside its tangency point $\gamma_j(t_j)$, intersection points deform continously in a neighborhood of $t_j$. It follows that $f_j$ is locally constant near $t_j$. Thus, $f_j$ is locally constant outside the values of $t_j$ where $\gamma_j(t_j)$ is a flex, a node, or $D(\gamma_j(t_j))$ is a bitangent. Let $J_j(t)=f_j(t_j^+)-f(t_j^-)$ be the value of the jump at $t_j$, where $t_j^\pm$ denotes the limit on the two sides of $t$. The jump function $J_j$ takes the following values, illustrated on Figure \ref{fig-walls-FB}:
\begin{itemize}[label=$\circ$]
\item Assume $\gamma_j(t_j)$ is a node. When crossing the node, passing from $t_j^-$ to $t_j^+$, one intersection point passes from $D_+(\gamma_j(t_j^-))$ to $D_-(\gamma_j(t_j^+))$, as depicted on Figure \ref{fig-walls-FB}(a). Thus, we have $J_j(t_j)=-2$.
\item Assume $\gamma_j(t_j)$ is a flex of $C_j$. Here, when crossing the flex, we also have an intersection point passing from $D_+(\gamma_j(t_j^-))$ to $D_-(\gamma_j(t_j^+))$, as depicted on Figure \ref{fig-walls-FB}(b), so that $J_j(t_j)=-2$.
\item Assume that the tangent $D(\gamma_j(t_j))$ is a bitangent: $D(\gamma_j(t_j))=D(\gamma_k(t'_k))$ for some $k$ and $t'_k$. Assume this bitangent is of type $T$. When passing from $t_j^-$ to $t_j^+$, we have two possibilities:
	\begin{itemize}[label=-]
	\item If $\gamma_k(t'_k)\in D_+(\gamma_j(t_j))$, we gain two intersection points in $D_+(\gamma_j(t_j))$, as depicted on Figure \ref{fig-walls-FB}(c).
	\item If $\gamma(t'_k)\in D_-(\gamma_j(t_j))$, we lose two intersection points in $D_-(\gamma_j(t_j))$, which would be depicted on Figure \ref{fig-walls-FB}(c) reversing the orientation of the arc.
	\end{itemize}
	In any case, we get that $J_j(t_j)=2$.
\item Assuming the bitangent is of type $S$, we get similarly $J_j(t_j)=-2$, as shown on Figure \ref{fig-walls-FB}(d).
\end{itemize}
As for each component we have $f_j(0)=f_j(1)$, we have
$$0=\sum_j f_j(1)-f_j(0)=\sum_j\sum_{t_j:J_j(t_j)\neq 0}J_j(t_j).$$
Grouping together the parameters giving the same node or bitangent, noticing that a bitangent appears twice (one per tangency point), we deduce the following relation:
$$4t(C)-4s(C)-4n(C)-2i(C) = 0,$$
which is the desired identity.
\end{proof}

\begin{figure}
    \centering
    \begin{tabular}{cc}
\begin{tikzpicture}[line cap=round,line join=round,>=triangle 45,x=0.4cm,y=0.4cm]
\clip(-8,-4) rectangle (8,4);
\draw [>->,thick,domain=-4:4, samples=100] 
 plot ({\x}, {0.25*\x*\x} );
\draw [thick,domain=-2:4, samples=100] 
 plot ({-0.3}, {\x} );
\draw (-2,1) node {$\bullet$};
\draw (2,1) node {$\bullet$};
\draw[red,->] (-2,1) --++(3,-3);
\draw[blue,-<] (-2,1)  --++(-2,2);
\draw[red,->] (2,1)--++(2,2);
\draw[blue,-<] (2,1)--++(-3,-3);
\end{tikzpicture} & \begin{tikzpicture}[line cap=round,line join=round,>=triangle 45,x=0.4cm,y=0.4cm]
\clip(-8,-4) rectangle (8,4);
\draw [>->,thick,domain=-4:4, samples=100] 
 plot ({\x}, {\x*\x*\x/16} );
\draw (-1.5,{-1.5^3/16}) node {$\bullet$};
\draw (1.5,{1.5^3/16}) node {$\bullet$};
\draw[red,->] (-1.5,{-1.5^3/16}) --++(6,{6*3*(1.5)^2/16});
\draw[blue,-<] (-1.5,{-1.5^3/16}) --++(-2,{-2*3*(1.5)^2/16});
\draw[red,->] (1.5,{1.5^3/16})--++(2,{2*3*(1.5)^2/16});
\draw[blue,-<] (1.5,{1.5^3/16})--++(-6,{-6*3*(1.5)^2/16});
\end{tikzpicture} \\
    (a) Passing through a node & (b) Passing through a flex \\
    \begin{tikzpicture}[line cap=round,line join=round,>=triangle 45,x=0.4cm,y=0.4cm]
\clip(-8,-4) rectangle (8,4);
\draw [>->,thick,domain=-3:3, samples=100] 
 plot ({\x-4}, {0.25*\x*\x} );
\draw [thick,domain=-3:3, samples=100] 
 plot ({\x+4}, {0.25*\x*\x} );
\draw ({-0.25-4},{0.25*(-0.25)^2}) node {$\bullet$};
\draw ({0.25-4},{0.25*(0.25)^2}) node {$\bullet$};
\draw[red,->] ({-0.25-4},{0.25*(-0.25)^2}) --++(12,{12*0.5*(-0.25)});
\draw[blue,-<] ({-0.25-4},{0.25*(-0.25)^2}) --++(-2,{-2*0.5*(-0.25)});
\draw[red,->] ({0.25-4},{0.25*(0.25)^2}) --++(11,{11*0.5*(0.25)});
\draw[blue,-<] ({0.25-4},{0.25*(0.25)^2}) --++(-2,{-2*0.5*(0.25)});
\end{tikzpicture} & \begin{tikzpicture}[line cap=round,line join=round,>=triangle 45,x=0.4cm,y=0.4cm]
\clip(-8,-4) rectangle (8,4);
\draw [>->,thick,domain=-3:3, samples=100] 
 plot ({\x-4}, {0.25*\x*\x} );
\draw [thick,domain=-3:3, samples=100] 
 plot ({\x+4}, {-0.25*\x*\x} );
\draw ({-0.25-4},{0.25*(-0.25)^2}) node {$\bullet$};
\draw ({0.25-4},{0.25*(0.25)^2}) node {$\bullet$};
\draw[red,->] ({-0.25-4},{0.25*(-0.25)^2}) --++(12,{12*0.5*(-0.25)});
\draw[blue,-<] ({-0.25-4},{0.25*(-0.25)^2}) --++(-2,{-2*0.5*(-0.25)});
\draw[red,->] ({0.25-4},{0.25*(0.25)^2}) --++(11,{11*0.5*(0.25)});
\draw[blue,-<] ({0.25-4},{0.25*(0.25)^2}) --++(-2,{-2*0.5*(0.25)});
\end{tikzpicture} \\
    (c) Passing through a type T bitangent & (d) Passing through a type S bitangent \\
    \end{tabular}
    \caption{Illustration of the values of the jump function at nodes, flexes and bitangents. The positive half-tangent is depicted in red, and the negative in blue. The alternative cases of passing through bitangents of types S and T are obtained by reversing the orientation of the curve and its tangents.}
    \label{fig-walls-FB}
\end{figure}

\section{A projective Fabricius-Bjerre Theorem}

We now aim to generalize the result by considering curves in $\RR P^2$. To do so, let $L_\infty$ be a line in $\RR P^2$ and consider curves generically transverse to $L_\infty$, which means the following: no line passing through a point of $C\cap L_\infty$ is a bitangent or tangent at a flex. Recall that $\TTT_\infty(C)$ is the set of tangents to $C$ at points of $C\cap L_\infty$. As $\RR P^2\backslash L_\infty=\RR^2$, we can use the Fabricius-Bjerre sign in this affine chart and make the associated signed count of bitangents. All constructions depend on the choice of this line $L_\infty$. If $L$ is a line distinct from $L_\infty$, we denote by $p_L$ the unique intersection point between $L$ and $L_\infty$.

%Let $\gamma:[0;1]\to\RR P^2$ be a parametrization of a component if $C$. We still denote by $D(\gamma(t))$ and $D_\pm(\gamma(t))$ the tangent and half-tangents at $\gamma(t)$. 

Let $L$ be a second line in $\RR P^2$, different from $L_\infty$. Before getting to the generalization of Theorem \ref{theo-affine-case}, we provide a signed count of the tangents to $C$ passing through $p_L$. Let $D$ be a tangent to $C$ passing through $p_L$. The set $\RR P^2\setminus(D\cup L_\infty\cup L)$ has three connected components, each one adjacent to two of the lines. Let $U$ the connected component in which $C$ is contained near its tangency point with $D$: $\overline{U}\supset D$. We set the following signs, depicted on Figure \ref{fig-signs-pencils}.:
    \begin{itemize}[label=$\circ$]
        \item if $\overline{U}\supset L$ (the concavity of $C$ is toward $L$), the sign is $\epsilon_L(D)=+1$ (cases $D_2$ and $D_3$);
        \item if $\overline{U}\supset L_\infty$ (the concavity of $C$ is toward $L_\infty$), the sign is $\epsilon_L(D)=-1$ (cases $D_1$ and $D_4$).
    \end{itemize}

\begin{figure}
    \centering
   \begin{tikzpicture}[line cap=round,line join=round,>=triangle 45,x=0.6cm,y=0.6cm]
\clip(-2,-6) rectangle (8,6);

\draw[thick,dashed] (0,-5) to (0,5); \draw (0,3) node[left] {$L_\infty$};
\draw[thick] (-1,0) to (6,0);
\draw (5,0) node[below] {$L$};
\draw (0,0) node[below left] {$p_L$};

\draw (4,2) node {$\bullet$};
\draw[dotted] (4,2)--++(2,1);
\draw[dotted] (4,2)--++(-2,-1);
\draw (6,3) node[right] {$D_2:+$};
\draw [thick,domain=-0.5:0.5, samples=100] 
 plot ({4+2*\x-1*(-1)*\x*\x}, {2+1*\x+2*(-1)*\x*\x} );
 
\draw (2,4) node {$\bullet$};
\draw[dotted] (2,4)--++(1,2);
\draw[dotted] (2,4)--++(-1,-2);
\draw (2.5,5) node[right] {$D_1:-$};
\draw [thick,domain=-0.5:0.5, samples=100] 
 plot ({2+1*\x-2*(1)*\x*\x}, {4+2*\x+1*(1)*\x*\x} );

\draw (4,-2) node {$\bullet$};
\draw[dotted] (4,-2)--++(2,-1);
\draw[dotted] (4,-2)--++(-2,1);
\draw (6,-3) node[right] {$D_3:+$};
\draw [thick,domain=-0.5:0.5, samples=100] 
 plot ({4-2*\x-1*(-1)*\x*\x}, {-2+1*\x-2*(-1)*\x*\x} );

\draw (2,-4) node {$\bullet$};
\draw[dotted] (2,-4)--++(1,-2);
\draw[dotted] (2,-4)--++(-1,2);
\draw (2.5,-5) node[right] {$D_4:-$};
\draw [thick,domain=-0.5:0.5, samples=100] 
 plot ({2+1*\x-2*(1)*\x*\x}, {-4-2*\x-1*(1)*\x*\x} );
\end{tikzpicture}
    \caption{Description signs of tangents passing through $p_L$.}
    \label{fig-signs-pencils}
\end{figure}
    
If $L$ is itself tangent to $C$, we set $\epsilon_L(L)=0$. We consider the signed count of tangents to $C$ passing through $p_L$:
$$\sigma_L(C)=\sum_{D\ni p_L}\epsilon_L(D).$$

\begin{lem}\label{lem-signed-count-asymptotic-line}
Assume $L\cap C\cap L_\infty=\emptyset$. The signed count of tangents to $C$ passing through $p_L$ has the following value:
$$\sigma_L(C)=|C\cap L|-|C\cap L_\infty| .$$
Furthermore, if $L\in\TTT_\infty(C)$ is tangent of $C$ at some point of $C\cap L_\infty$, then
$$\sigma_L(C)=1+|C\cap L|-|C\cap L_\infty| .$$
\end{lem}

\begin{proof}
    First assume that $p_L\notin C$. We consider the projection from $C$ to the pencil $\P\simeq\RR P^1$ of lines passing through $p_L$:
    $$\pi:p\in C\longmapsto (pp_L)\in\P,$$
    where $(pp_L)$ is the unique line passing through $p$ and $p_L$. The pencil $\P$ contains the lines $L$ and $L_\infty$, which split $\P$ into segments $\P_0$ and $\P_1$, which we both orient from $L$ to $L_\infty$.

    A coordinate function on $\P_0$ (resp. $\P_1$) induces a Morse function on $\pi^{-1}(\P_0)\subset C$. Its critical points correspond precisely to points in $C$ whose tangent belongs to $\P$, and the sign $\epsilon_L$ actually coincides with the index. Therefore, for $\P_0$, we have that
    $$|\pi^{-1}(L)|-|\pi^{-1}(L_\infty)| = 2\sum_{D\in\P_0}\epsilon_L(D),$$
    where the sum is over the tangent to $C$ belonging to $\P_0$. As $\pi^{-1}(L)=C\cap L$ and $\pi^{-1}(L_\infty)=C\cap L_\infty$, adding the identities for $\P_0$ and $\P_1$ yields the count $\sigma_L(C)$ and its value.

    If $L=T$ is the tangent line to a point $p_T\in C\cap L_\infty$, the projection $\pi$ is not a priori defined at $p_T$. As $C$ is tangent to $T$ at $p_T$, $\pi$ extends by continuity with $\pi(p_T)=T\in\P$. We then proceed as before, with the only difference that $\pi^{-1}(L_\infty)=(C\cap L_\infty)\backslash\{p_T\}$, yielding the $1$ discrepancy.
\end{proof}

We can now give the projective version of Fabricius-Bjerre.

\begin{figure}
    \centering
\begin{tikzpicture}[line cap=round,line join=round,>=triangle 45,x=0.4cm,y=0.4cm]
\clip(-8,-4) rectangle (8,4);
\draw [>->,thick,domain=-4:4, samples=100] 
 plot ({\x}, {0.25*\x*\x} );
\draw[thick,dashed] (0,-3) to (0,4); \draw (0,3) node[right] {$L_\infty$};
\draw (-2,1) node {$\bullet$};
\draw (2,1) node {$\bullet$};
\draw[red] (-2,1) --++(2,-2);
\draw[blue,->] (0,-1) --++(2,-2);
\draw[blue,-<] (-2,1)  --++(-2,2);
\draw[red,->] (2,1)--++(2,2);
\draw[blue] (2,1)--++(-2,-2);
\draw[red,-<] (0,-1)--++(-2,-2);
\end{tikzpicture}
    \caption{Change of value of the jump function when the curve intersects $L_\infty$: intersection points change from blue ($D_-$) to red part ($D_+$).}
    \label{fig-walls-FB-proj}
\end{figure}

\begin{theo}\label{theo-projective-case}
Let $C$ be a curve in $\RR P^2$ transverse to $L_\infty$ with $a$ intersection points. We have the following signed count of bitangents:
$$\sigma(C)=n(C)+\frac{i(C)}{2}+\frac{a(a-2)}{2}-\sum_{T\in\TTT_\infty(C)} (|C\cap T|-1) .$$
\end{theo}

\begin{proof}
We proceed as in the affine case, considering a parametrization $\gamma_j:I_j\to\RR P^2$ of each component of $C$ and the functions
$$f_j(t_j)=|D_+(\gamma_j(t_j))\cap C|-|D_-(\gamma_j(t_j))\cap C|.$$
Outside a finite number of values, $f_j$ is locally constant. Let $J_j(t_j)=f(t_j^+)-f(t_j^-)$ be the jump function. On one side, we have the same jumps as for the compact curves:
	\begin{itemize}[label=$\circ$]
	\item $J_j(t_j)=2$ if $\gamma_j(t_j)$ is node or a flex, or if $D(\gamma_j(t_j))$ is a bitangent of type $T$;
	\item $J_j(t_j)=-2$ if $D(\gamma_j(t_j))$ is a bitangent of type $S$.
	\end{itemize}
However, we have now new jump points when $\gamma_j(t_j)\in L_\infty$ or if $D(\gamma_j(t_j))\ni p_T$ for $T\in\TTT_\infty(C)$, as depicted on Figure \ref{fig-walls-FB-proj}.
\begin{itemize}[label=$\circ$]	
\item Assume $\gamma_j(t_j)\in L_\infty$. In particular, $D(\gamma_j(t_j))=T\in\TTT_\infty(C)$. Going from $t_j^-$ to $t_j^+$, all points of $C\cap T-\{p_T\}$ move from $D_-$ to $D_+$. Therefore, we have $J_j(t_j)=2(|C\cap T|-1)$. The $-1$ is for the tangency point, not counted in the process. On Figure \ref{fig-walls-FB-proj}, we see that the tangent goes from (almost) fully blue to fully red.
\item Assume that $D(\gamma_j(t_j))\ni p_T$ for some $T\in\TTT_\infty(C)$. Depending on the concavity of $C$ at $\gamma_j(t_j)$, the intersection point near $p_T$ passes from $D_+(\gamma_j(t_j^-))$ to $D_-(\gamma_j(t_j^+))$ or conversely, as can be seen on Figure \ref{fig-wall-FB-proj-2}. We thus have $J_j(t_j)=2\epsilon_T(D(\gamma_j(t_j)))=\pm 2$.
\end{itemize}
As $0=f_j(1)-f_j(0)=\sum_{t_j:J_j(t)\neq 0}J_j(t_j)$, summing over the various components and grouping together the parameters corresponding to bitangents and nodes, we get the following identity:
$$4t(C)-4s(C)-2i(C)-4n(C) + 2\sum_{T\in\TTT_\infty(C)}(|T\cap C|-1) + \sum_{T\in\TTT_\infty(C)}\sum_{D\ni p_T}2\epsilon_T(D)=0.$$
We recognize the signed count $\sigma_T(C)$ computed in Lemma \ref{lem-signed-count-asymptotic-line}. Replacing its expression, we get that
$$4\sigma(C)-2i(C)-4n(C) +4\sum(|T\cap C|-1) -2a(a-2)=0.$$
\end{proof}

\begin{figure}
    \centering
\begin{tabular}{cc}
    \begin{tikzpicture}[line cap=round,line join=round,>=triangle 45,x=0.5cm,y=0.5cm]
\clip(-2,-2) rectangle (6,4);

\draw[thick,dashed] (0,-5) to (0,5); \draw (0,3) node[left] {$L_\infty$};
\draw[thick,domain=-4:4, samples=100] 
 plot ({3+\x}, {-2+2*\x*\x/9} );
\draw (0,0) node[below left] {$p_L$};

\draw (4,2) node {$\bullet$};
\draw[dotted] (4,2)--++(2,1);
\draw[dotted] (4,2)--(-6,-3);
\draw [>->,thick,domain=-0.7:0.7, samples=100] 
 plot ({4+2*\x-1*(-1)*\x*\x}, {2+1*\x+2*(-1)*\x*\x} );

\draw[blue,>-] (0,-1)--(4,2);
\draw[red,->] (4,2)--++({4*0.4},{3*0.4});
\draw[red,-<] (0,-1)--++({-4*0.3},{-3*0.3});

\draw[blue,>-] (0,1)--(4,2);
\draw[red,->] (4,2)--++({4*0.4},{0.4});
\draw[red,-<] (0,1)--++({-4*0.3},{-0.3});
\end{tikzpicture} &  \begin{tikzpicture}[line cap=round,line join=round,>=triangle 45,x=0.5cm,y=0.5cm]
\clip(-2,-2) rectangle (6,4);

\draw[thick,dashed] (0,-5) to (0,5); \draw (0,3) node[left] {$L_\infty$};
\draw[thick,domain=-4:4, samples=100] 
 plot ({3+\x}, {-2+2*\x*\x/9} );
\draw (0,0) node[below left] {$p_L$};

\draw (4,2) node {$\bullet$};
\draw[dotted] (4,2)--++(2,1);
\draw[dotted] (4,2)--(-6,-3);
\draw [<-<,thick,domain=-0.7:0.7, samples=100] 
 plot ({4+2*\x-1*(-1)*\x*\x}, {2+1*\x+2*(-1)*\x*\x} );

\draw[red,<-] (0,-1)--(4,2);
\draw[blue,-<] (4,2)--++({4*0.4},{3*0.4});
\draw[blue,->] (0,-1)--++({-4*0.3},{-3*0.3});

\draw[red,<-] (0,1)--(4,2);
\draw[blue,-<] (4,2)--++({4*0.4},{0.4});
\draw[blue,->] (0,1)--++({-4*0.3},{-0.3});
\end{tikzpicture} \\
   (a) intersection point goes from blue to red  & (b) intersection point goes from blue to red \\
\begin{tikzpicture}[line cap=round,line join=round,>=triangle 45,x=0.5cm,y=0.5cm]
\clip(-2,-2) rectangle (6,4);

\draw[thick,dashed] (0,-5) to (0,5); \draw (0,3) node[left] {$L_\infty$};
\draw[thick,domain=-4:4, samples=100] 
 plot ({3+\x}, {-2+2*\x*\x/9} );
\draw (0,0) node[below left] {$p_L$};

\draw (4,2) node {$\bullet$};
\draw[dotted] (4,2)--++(2,1);
\draw[dotted] (4,2)--(-6,-3);
\draw [>->,thick,domain=-0.7:0.7, samples=100] 
 plot ({4+2*\x+1*(-1)*\x*\x}, {2+1*\x+2*\x*\x} );

\draw[blue,>-] (0,-1)--(4,2);
\draw[red,->] (4,2)--++({4*0.4},{3*0.4});
\draw[red,-<] (0,-1)--++({-4*0.3},{-3*0.3});

\draw[blue,>-] (0,1)--(4,2);
\draw[red,->] (4,2)--++({4*0.4},{0.4});
\draw[red,-<] (0,1)--++({-4*0.3},{-0.3});
\end{tikzpicture} &  \begin{tikzpicture}[line cap=round,line join=round,>=triangle 45,x=0.5cm,y=0.5cm]
\clip(-2,-2) rectangle (6,4);

\draw[thick,dashed] (0,-5) to (0,5); \draw (0,3) node[left] {$L_\infty$};
\draw[thick,domain=-4:4, samples=100] 
 plot ({3+\x}, {-2+2*\x*\x/9} );
\draw (0,0) node[below left] {$p_L$};

\draw (4,2) node {$\bullet$};
\draw[dotted] (4,2)--++(2,1);
\draw[dotted] (4,2)--(-6,-3);
\draw [<-<,thick,domain=-0.7:0.7, samples=100] 
 plot ({4+2*\x+1*(-1)*\x*\x}, {2+1*\x+2*\x*\x} );

\draw[red,<-] (0,-1)--(4,2);
\draw[blue,-<] (4,2)--++({4*0.4},{3*0.4});
\draw[blue,->] (0,-1)--++({-4*0.3},{-3*0.3});

\draw[red,<-] (0,1)--(4,2);
\draw[blue,-<] (4,2)--++({4*0.4},{0.4});
\draw[blue,->] (0,1)--++({-4*0.3},{-0.3});
\end{tikzpicture} \\
   (c) intersection point goes from red to blue  & (d) intersection point goes from red to blue
\end{tabular}
    \caption{Change of value of the jump function when a tangent hits a point of $C\cap L_\infty$: if concavity is toward $T$ (case (a) and (b)), intersection point near $p_L$ goes from blue to red, and if concavity is toward $L_\infty$ (case (c) and (d)), intersection point near $p_L$ goes from red to blue.}
    \label{fig-wall-FB-proj-2}
\end{figure}

\section{Applications to real algebraic curves}

Let $\CCC$ be a generic smooth real algebraic curve in $\RR P^2$ of degree $d\geqslant 2$. In particular, the real part $\RR\CCC$ is a curve in our sense. If $L$ is a real line tangent to $\CCC$ at two points, there are two possibilities:
	\begin{itemize}[label=$\circ$]
	\item the \textit{split} bitangents if the tangency points are real,
	\item the \textit{non-split} bitangents if the tangency points are complex conjugate.
	\end{itemize}
Actually, the number of non-split bitangents $t_0(\CCC)$ is related to the number of real flexes $i(\RR\CCC)$ via Klein's formula \cite{klein1876formula}:
$$t_0(\CCC)+\frac{i(\RR\CCC)}{2}=\frac{d(d-2)}{2}.$$
We consider the following signed count of bitangents to $\CCC$:
$$\rho(\CCC)=t_0(\CCC)+\sigma(\RR\CCC).$$
The first term deals with non-split bitangents (thus counted with a positive sign), and the second term with the split bitangents, counted with the Fabricius-Bjerre sign. The curve being smooth, we have $n(\RR\CCC)=0$. Coupling Klein's formula with Theorem \ref{theo-projective-case}, we get the following.

\begin{theo}\label{theo-algebraic-case}
Let $\CCC$ be a generic smooth real algebraic curve of degree $d$ in $\RR P^2$ with $a$ intersection points with $L_\infty$. Then, we have the following signed count:
$$\rho(\CCC)=\frac{d(d-2)}{2}+\frac{a(a-2)}{2}-\sum_{T\in\TTT_\infty(\RR\CCC)}(|\RR\CCC\cap T|-1) .$$
Furthermore, this count is even and non-negative.
\end{theo}

Theorem \ref{theo-algebraic-case} actually proves \cite[Conjecture 2]{larsonvogt2021enriched} and generalizes it for curves of any degree. This conjecture was already proved by Kummer-McKean \cite{kummer2024bounding} by a different method in the case of quartics.

\begin{proof}
As advertised, the signed count is a consequence of Theorem \ref{theo-projective-case} and Klein's formula. We now prove the count is even and non-negative. By Bezout's theorem, for each tangent $T$, we have that
$$|T\cap\RR\CCC|-1\leqslant d-2.$$
Summing over all the tangents of $\TTT_\infty(\RR\CCC)$, we have the following lower bound:
$$\rho(\CCC) \geqslant \frac{d(d-2)}{2}+\frac{a(a-2)}{2}-a(d-2) = \frac{(d-a)(d-a-2)}{2}\geqslant 0,$$
since by Bezout's theorem, $d-a\neq 1$. For the parity, Bezout's Theorem also tells us that $d\equiv a\mod 2$ and $|T\cap\RR\CCC|-1\equiv d\modulo 2$ since the tangency points count twice. Therefore,
\begin{align*}
\rho(\CCC) \equiv & \frac{a^2+d^2}{2}-a-d-ad \modulo 2 \\
\equiv & \frac{a^2-d^2}{2} \modulo 2 .
\end{align*}
As $a\equiv d\modulo 2$, we get $a^2\equiv d^2 \modulo 4$ and we conclude.
\end{proof}

Assuming $\CCC$ has no real point on $L_\infty$, \textit{i.e.} $a=0$, the signed count in this particular case is
$$\rho(\CCC)=\frac{d(d-2)}{2}.$$

\begin{expl}
For conics, \textit{i.e.} $d=2$, the tangents at $\RR\CCC\cap L_\infty$ have no further intersection points, so that the signed count is $0+0-0$, compatible with the fact there is no bitangent.

We now assume $\CCC$ is a cubic curve: the degree is $3$. Klein's formula tells us that there are $3$ real flexes. The number of tangents at $\RR\CCC\cap L_\infty$ is either $1$ or $3$. Furthermore, for any real tangent line $L$, by Bezout's theorem, $L\cap\CCC$ has three points counted with multiplicity. The tangency point counts for $2$. Thus, the remaining point is real, and we get $|L\cap\RR\CCC|-1=1$. In the end, if $a=1$ we get $\frac{3}{2}-\frac{1}{2}-1$, and if $a=3$ we get $\frac{3}{2}+\frac{3}{2}-3$, which is compatible with the absence of bitangent in this case as well.
\end{expl}

\begin{expl}
The first interesting case is the case of quartic curves. Assuming transversality, a quartic may have $a=0,2$ or $4$ real intersection points with $L_\infty$.
	\begin{itemize}[label=$\triangleright$]
	\item If $a=0$, the real locus is compact in $\RR^2$, and we recover the signed count from \cite{larsonvogt2021enriched}.
	\item If $r=2$, we have tangents at points of $L_\infty$. The signed count is
$$\rho(\CCC)=4-(|T_1\cap \RR\CCC|-1)-(|T_2\cap \RR\CCC|-1) .$$
Thanks to Bezout's theorem, we have that $|T\cap \RR\CCC|-1\leqslant 2$ for every tangent $T$ to $C$. Thus, $\rho(\CCC)\geqslant 0$ and the count may take the values $0,2,4$.
	\item If $r=4$, we now a have $4$ asymptotic lines, and we get
$$\rho(\CCC) = 4+\frac{4(4-2)}{2}-\sum_1^4 (|T_j\cap\RR\CCC|-1) = 8-\sum_1^4 (|T_j\cap\RR\CCC|-1).$$
We conclude similarly than $\rho(\CCC)\geqslant 0$ and can take the values $0,2,4,6,8$.
	\end{itemize}
\end{expl}

Using the generalization of Klein's formula \cite{Schuh} to curves with nodes and cusp, we can generalize the signed count to real algebraic curves with nodes.

\begin{prop}
    Let $\CCC$ be a degree $d$ real nodal algebraic curve in $\PP^2$ with $N$ nodes (both real and complex), $n_\RR$ real nodes, $i$ real flexes and $a$ intersection points with $L_\infty$. Then, we have the following signed count:
    $$\rho(\CCC) = n_\RR-N + \frac{d(d-2)}{2} + \frac{a(a-2)}{2}-\sum_{T}(|C\cap T|-1)$$
\end{prop}

\begin{proof}
    Let $n_0$ be the number of real isolated nodes and $n_2$ the number of hyperbolic nodes, so that $n_\RR=n_0+n_2$. Klein's formula for nodal curves \cite{Schuh} states as follows:
    $$ d+i+2t_0 = d(d-1)-2N+2n_0,$$
    while Theorem \ref{theo-projective-case} gives
    $$\sigma(\RR\CCC) = n_2+\frac{i}{2} + \frac{a(a-2)}{2}-\sum_{T}(|C\cap T|-1).$$
    We conclude as in Theorem \ref{theo-algebraic-case}.
\end{proof}

\bibliographystyle{alpha}
\bibliography{biblio}

\end{document}